\documentclass{amsart}

\usepackage{amsmath,tabu}
\usepackage{amsfonts}
\usepackage{amssymb}
\usepackage{amsthm}
\usepackage[bookmarks]{hyperref}
\hypersetup{colorlinks=false}
\usepackage{cleveref}
\usepackage{setspace}
\usepackage{enumerate}

\newcommand{\mS}{\mathcal{S}}
\newcommand{\mT}{\mathcal{T}}

\newcommand{\M}{\mathbb{M}}
\newcommand{\ot}{\otimes}
\newcommand{\ra}{\rightarrow}
\newcommand{\C}{\mathbb{C}}

\newcommand{\B}{B(\mathcal{H})}

\newcommand{\mH}{\mathcal{H}}

\newcommand{\N}{\mathbb{N}}
\newcommand{\F}{\mathbb{F}}
\newcommand{\mC}{\mathcal{C}}
\newcommand{\D}{\mathcal{D}}

\newcommand{\mO}{\mathcal{O}}

\newcommand{\ep}{\varepsilon}

\newcommand{\mF}{\mathfrak{u}}
\newcommand{\mV}{\mathfrak{v}}
\newcommand{\Z}{\mathbb{Z}}

\newcommand{\mcal}{\mathcal}

\theoremstyle{thmit}
\newtheorem{theorem}{Theorem}[section]
\newtheorem{definition}[theorem]{Definition}
\newtheorem{proposition}[theorem]{Proposition}

\newtheorem{corollary}[theorem]{Corollary}

\theoremstyle{thmrm}
\newtheorem{example}[theorem]{Example}
\newtheorem{remarks}[theorem]{Remark}

\title{{ Operator System Nuclearity via $C^*$-envelopes}}
\date{\today}

\author[Ved Gupta]{Ved Prakash Gupta}
\address{School of Physical Sciences\\ Jawaharlal Nehru University\\ New Delhi-110067, INDIA}
\email{vedgupta@mail.jnu.ac.in}

\author[Preeti Luthra]{Preeti Luthra}
\address{Department of Mathematics\\ University of Delhi\\ Delhi-110007, INDIA}
\email{maths.preeti@gmail.com}

\keywords{Operator systems, amenable groups, graphs, nuclearity,
  $C^*$-envelopes, tensor products.\\ 
\noindent\textit{Mathematics Subject Classification (2010): Primary
  46L06, 46L07; Secondary 47L25, 46M05, 46B28}}

\thanks{The first named author was supported
  partially by a Start-Up grant (Grant No. F. 30-51/2014(BSR)) of the
  University Grants Commission (UGC), India and the second named
  author was supported partially by a Junior Research Fellowship
  (Grant No. 09/045(1133)/2011-EMR-I) of Council of Scientific and
  Industrial Research (CSIR), India.}

\begin{document}
\maketitle

\begin{abstract}
We prove that an operator system is (min, ess)-nuclear if its $C^*$-envelope is
nuclear.  This allows us to deduce that an operator system associated
to a generating set of countable discrete group by Farenick et al.~is
(min, ess)-nuclear if and only if the group is amenable.  We also make
a detailed comparison between ess and other operator system tensor
products and show that an operator system associated to a minimal
generating set of a finitely generated discrete group (resp., a finite graph) is (min,
max)-nuclear if and only if the group is of order less than
or equal to $3$ (resp., every component of the graph is complete). 
\end{abstract}

\section{Introduction}\label{s:1}
An operator system is a self-adjoint unital subspace of $\B$ for some
complex Hilbert space $\mathcal{H}$. Choi and Effros (\cite{choi})
obtained an abstract characterization of an operator system and quite
recently this abstraction proved very useful in the development of the
theory of tensor products in the category of operator systems in a
series of papers \cite{KPTT1, KPTT2, han, opsysqt, weakexp, kavnuc}. A
short survey of this development is available in Chapter $4$ of
\cite{faqit}.  Essentially, a lattice of five tensor products of
operator systems consisting of the so called {\em minimal} (min), the
{\em maximal} (max), the {\em maximal commuting} (c), the {\em left
  enveloping} (el) and the {\em right enveloping} (er) tensor products
were introduced in \cite{KPTT1} and unlike the category of
$C^*$-algebras, a variety of nuclearity considerations were made,
{viz.}, given two operator system tensor products $ \alpha$ and
$\beta$, an operator system $\mS$ is said to be $(\alpha,
\beta)$-nuclear if $\mS \ot_{\alpha} \mT = \mS \ot_{\beta} \mT$ for
every operator system $\mT$. In last few years, various
characterizations of these notions of nuclearity among the primary
tensor products in terms of some intrinsic properties of operator
systems, {viz.}, {\em Exactness}, {\em Weak Expectation Property
  (WEP)}, {\em Double Commutant Expectation Property (DCEP)}, {\em
  Operator System Local Lifting Property (OSLLP)} and {\em Completely
  Positive Factorization Property (CPFP)} have been established (
\cite{KPTT1, KPTT2, han, opsysqt, weakexp, kavnuc}); also, it is known
that the maximum nuclearity expected for a finite dimensional operator
system which is not isomorphic to a $C^*$-algebra is (min,
c)-nuclearity, equivalently, $C^*$-nuclearity  (\cite{KPTT1, kavnuc}).

Very recently, Farenick {et al.}~in \cite{disgrp} associated an operator
system $\mS(\mF)$ to every generating set $\mathfrak{u}$ of a discrete
group $G$ and based on the developments of
\cite{opsysqt,weakexp,kavnuc} they could relate the longstanding
operator algebraic questions of {\em Connes' embedding problem} and
{\em Kirchberg's conjecture} to the tensor products of such operator
systems associated to the free group with two generators. In the same
paper they also introduced and initiated the study of the natural operator system
tensor product ``ess'' arising from the enveloping $C^*$-algebras,
viz., $$\mathrm{\mS \ot_{ess} \mT \subseteq C^*_e(\mS) \ot_{max} C^*_e(\mT)}.$$

It was shown in \cite[$\S 10$]{kavnuc} that the $(2n+1)$-dimensional
operator system $\mS_n$ associated to the universal generating set of
the free group with $n$ generators, for $n \geq 2$, is not exact and
hence $\mS_n$ is not $C^*$-nuclear. This also illustrates that unlike
$C^*$-algebras, finite dimensional operator systems need not be
nuclear.

This paper came into existence out of the curiosity to understand
nuclearity properties of the operator systems associated to amenable
discrete groups. Lance proved that the group $C^*$-algebra of a
discrete group is nuclear if and only if the group is amenable
(\cite{lan}). Also, if $\mF$ is a generating set of a discrete group
$G$ and the associated operator system $\mS(\mF)$ is a $C^*$-nuclear
operator system then, by \cite[Corollary $9.6$]{KPTT2}, the group
$C^*$-algebra of $G$ is a nuclear $C^*$-algebra and, therefore, the
group $G$ is amenable. However, it is not yet clear (at least, to us)
whether the converse holds or not. Making some progress in this
direction, we deduce in Section $5$ that for any generating set $\mF$
of a discrete group $G$, $\mS(\mF)$ is $\mathrm{(min, ess)}$-nuclear
if and only if the group is amenable.

The tool that helps us to achieve this is the notion of the $C^*$-envelope
of an operator system which was introduced by Arveson in
\cite{arveson} and whose existence for an arbitrary operator system
was first established by Hamana in \cite{hamana2}.

   It is known that, in general, nuclearity does not pass to
   $C^*$-subalgebras (\cite[Remark 4.4.4]{brownozawa}). In
   particular, the same therefore holds for $\mathrm{(min, c)}$-nuclearity of
   operator systems.  The relationship between an operator system and
   its $C^*$-envelope is equally mysterious, which we highlight
   briefly at the beginning of Section $4$. Thereafter, we are able to
   see that the nuclearity of $C^*$-envelope behaves well if we
   restrict to $\mathrm{(min, ess)}$-nuclearity.  Moreover, we prove
   that if an operator system contains enough unitaries of its
   $C^*$-envelope then its $(\min, \mathrm{ess})$-nuclearity is
   equivalent to the nuclearity of its $C^*$-envelope.
 
After a short section on Preliminaries in Section $2$, we first make
some comparisons between the ``ess" operator system tensor product and
other operator system tensor products in Section $3$.

In Section $5$, apart from the characterization of $(\min,
\mathrm{ess})$-nuclearity of the group operator system $\mS(\mF)$ in terms of
amenability of the group, we also provide an exhaustive list of (min,
max)-nuclear operator systems associated to minimal  generating sets
of finitely generated groups.

In Section $6$, we characterize (min,max)-nuclear graph operator
systems (as introduced in \cite{KPTT1}) for finite graphs, purely, in terms of graph
theoretic properties. We achieve this characterization using an
identification of their $C^*$-envelopes obtained in \cite{OP}.

Finally, in Section $7$, as yet another application of our results on
$C^*$-envelopes, we  discuss nuclearity related properties of some known examples of
operator systems (from \cite{disgrp,  argerami1,
  argerami2}) whose $C^*$-envelopes are either known or whose nuclearity can be easily deduced.

\section{Preliminaries}\label{s:2}

For the basics on operator systems, we refer the reader
to \cite{paulsen, effros-ruan}. And, in order to avoid repetition, we
will freely borrow and follow terminologies and notations for tensor
products of operator systems and nuclearity related properties from
\cite{KPTT1,KPTT2,kavnuc}. However, for the sake of completeness we
include certain definitions and results that will be used
subsequently.

A \textit{$C^*$-cover} (\cite[$\S 2$]{hamana2}) of an operator system
$\mS$ is a pair $(A, i)$ consisting of a unital $C^*$-algebra $A$ and
a complete order embedding $i : \mS \rightarrow A$ such that $i(A)$
generates the $C^*$-algebra $A$. An \textit{essential extension}
(\cite[$\S 2$]{hamana2}) of an operator system $\mathcal{S}$ is a pair
$(\mathcal{W}, j)$ consisting of an operator system $\mathcal{W}$ and
a complete order embedding $j: \mS \ra \mathcal{W}$ such that for any
operator system $\mT$ and any unital completely positive map $\varphi:
\mathcal{W} \ra \mT$, $\varphi$ is a complete order embedding whenever
$\varphi \circ j$ is. The \textit{injective envelope} (\cite{choi}),
denoted by $I(\mS)$, of an operator system $\mS$ is the minimal
injective operator system that contains $\mS$. Its existence and
uniqueness was first proved by Hamana in \cite{hamana2}.

We will work with Hamana's version of $C^*$-envelope \cite{hamana2},
according to which the $C^*$-envelope of an operator system $\mS$ is a
$C^*$-cover defined as the $C^*$-algebra generated by $\mS$ in its
injective envelope $I(\mS)$ and is denoted by $C^*_e (\mS)$. It is
known that the injective envelope $I(\mS)$ is an essential extension
of $\mS$ (\cite[$\S 3$]{hamana2}) and the $C^*$-envelope $C^*_e (\mS)$
enjoys the following universal ``minimality'' property:

{\em Identifying $\mS$ with its image in $C^*_e (\mS)$, for any $C^*$-cover
$(A, i)$ of $\mS$, there is a unique surjective unital
$*$-homomorphism $\pi : A \ra C^*_e (\mS)$ such that $\pi(i(s)) = s$
for every $s$ in $\mS$ (\cite[Corollary 4.2]{hamana2}).}

The following, which is folklore, is an immediate consequence of above
universality and will be used quite often in the coming sections:

\begin{proposition}\label{nuclear-C-cover} 
If an operator system $\mS$ possesses a nuclear $C^*$-cover, then
$C^*_e(\mS)$ is nuclear.
\end{proposition}
\begin{proof}
Using the universal property of $C^*$-envelopes and the fact that
quotient of a nuclear $C^*$-algebra is nuclear (\cite[Corollary
  $9.4.4$]{brownozawa}) the statement follows.
\end{proof}

 We now list some useful rigidity properties (among which ($ii$) was also
 pointed out in \cite[$\S 1$]{kavnuc}) and an immediate consequence of
 the universality of $C^*_e (\mS)$, and, even though these are
 folklore, we provide their details for the sake of completeness. We would like to 
 thank K.~H.~Han for sharing  alternate proofs of $(i)$ and $(ii)$,
 which we have included here for their brevity.

\begin{proposition}\label{rigidity} 
For an operator system $\mS$, the following hold:
\begin{enumerate}[(i)]
\item Suppose $C^*_e(\mS) \subset \B$ and $\varphi: C^*_e(\mS)
  \rightarrow \B$ is a unital completely positive map that fixes
  $\mS$, then $\varphi$ fixes $C^*_e(\mS)$.
\item If $\psi : C^*_e (\mS) \ra \mT$ is a unital completely
  positive map into an  operator system $\mT$ such that
  $\psi_{|_{\mS}}$ is a complete order embedding then $\psi$ is a
  complete order embedding.
\item If $\mS$ is unitally completely order
  isomorphic to a unital $C^*$-algebra, then $\mS =
  C^*_e(\mS).$ \label{pi}
\end{enumerate}
\end{proposition}

\begin{proof}
 By injectivity of $B(H)$, $\varphi$  admits a unital
 completely positive extension $\tilde{\varphi} : I(\mS) \ra B(H)$ and
 for the same reason, the inclusion $\iota : C^*_e(\mS) \ra I(\mS)$
 also admits a unital completely positive extension $\tilde{\iota} :
 B(H) \ra I(\mS)$. The composition $\tilde{\iota} \circ
 \tilde{\varphi} : I(\mS) \ra I(\mS)$ clearly fixes $\mS$ and is unital completely positive,
 therefore, by ridigity of $I(\mS)$ (as in \cite[$\S 3$]{hamana2}), it
 fixes whole of $I(\mS)$; and since $\varphi$ agrees with $\tilde{\iota} \circ
 \tilde{\varphi}$ on $C^*_e(\mS)$, this proves $(i)$.

Let $\mT\subset B(H)$ for some Hilbert space $H$. Then, by injectivity
of $B(H)$, the map $\psi: C^*_e(\mS) \ra \mT \subset B(H)$ admits a
unital completely positive extension $\tilde{\psi}: I(\mS) \ra
B(H)$. Since $\tilde{\psi}_{|_{\mS}}$ is a complete order embedding
and $\mS \subset I(\mS)$ is an essential extension (\cite[$\S
  3$]{hamana2}), it follows that $\tilde{\psi}$ is a complete order
embedding and hence $(ii)$ follows.

For $(iii)$, let $\varphi : \mS \ra A$ be a unital complete order
isomorphism for some unital $C^*$-algebra $A$.  Then, $(A, \varphi)$
being a $C^*$-cover of $\mS$, by universality of $C^*_e(\mS)$,
there exists a surjective unital $*$-homomorphism $\pi:A \ra
C^*_e(\mS)$ such that $\pi \circ \varphi(s) = s$ for every $s \in \mS$
implying that $\mS = C^*_e(\mS)$.
\end{proof}

We must remark here that \Cref{rigidity}$(iii)$ turns out to be
the main ingredient in the classification of nuclear operator systems
associated to finitely generated groups and finite graphs. It also
allows us to identify some non-nuclear finite dimensional operator
systems as we will see in Section $7$.

An operator subsystem $\mS$ of a unital $C^*$-algebra $A$ is said to
contain enough unitaries of $A$ if the unitaries in $\mS$ generate $A$
as a $C^*$-algebra (\cite[$\S 9$]{KPTT2}). We will come across quite a
few instances where we will have to appeal to the following useful
result (\cite[Proposition $5.6$]{kavnuc}) of Kavruk:

\begin{proposition}\label{unitary-envelope}
Suppose $\mS \subset A$ contains enough unitaries. Then, upto a
$\ast$-isomorphism fixing $\mS$, we have $A = C^*_e(\mS)$.
\end{proposition}

In order to distinguish between operator system tensor products and
$C^*$-tensor products, we will use $\ot_{C^*\text{-}{\min}}$ and
$\ot_{C^*\text{-}{\max}}$ to denote the minimal and maximal
$C^*$-tensor products, respectively.
\begin{corollary} 
Let $\mS \subset A$ and $\mT \subset B$ be operator systems containing
enough unitaries of $A$ and $B$ respectively. Then, upto a
$\ast$-isomorphism fixing $\mS \ot_{\mathrm{ess}} \mT$, we have $C^*_e(\mS
\ot_{\mathrm{ess}} \mT)= A \ot_{C^*\text{-}{\max}} B$.
\end{corollary}

\begin{proof} 

By the definition of $\ot_{\text{ess}}$, we note that $\mS \ot_{\mathrm{ess}}
\mT$ contains enough unitaries of $C^*_e(\mS) \ot_{C^*\text{-}{\max}}
C^*_e(\mT)$ and hence, by \Cref{unitary-envelope}, there
does exist a desired $\ast$- isomorphism.
\end{proof}

Apart from the $C^*$-envelope of an operator system, there is one more
fundamental $C^*$-cover associated to an operator system $\mS$,
namely, the universal $C^*$-algebra $C^*_u(\mS)$ introduced by
Kirchberg and Wassermann (\cite[$\S 3$]{kirchberg}). We use this
notion and operator system techniques to provide a proof of the
following folklore result:

\begin{proposition}\label{non-nuclear}
If $A$ is a non-nuclear unital $C^*$-algebra, then there exists a
unital $C^*$-algebra $B$ such that there is no $\ast$-isomorphism
between $A \ot_{C^*\text{-}{\min}} B$ and $ A \ot_{C^*\text{-}{\max}}
B$ fixing $A \ot B$.
\end{proposition}

\begin{proof}
Since $A$ is a non-nuclear $C^*$-algebra, by \cite[Corollary
  6.8]{KPTT1}, there exists an operator system $\mS$ such that $A
\ot_{\min} \mS \neq A \ot_{\max} \mS$. Now, using injectivity of
$\ot_{\min}$, $A \ot_{\min} \mS \subseteq A \ot_{\min} C^*_u(\mS)$ and, by
definition of $\ot_c$ and the fact that $\ot_c$ coincides with $\ot_{\max}$ if
one of the factors in the tensor product is a $C^*$-algebra
(\cite[Theorem $6.7$]{KPTT1}), we have, $A \ot_{\max} \mS = A
\ot_{\mathrm{c}} \mS \subseteq A \ot_{\max} C^*_u(\mS)$, where the
last complete order embedding is guaranteed by \cite[Lemma
  $2.5$]{KPTT2}.  Thus, $A \ot_{\min} C_u^*(\mS) \neq A \ot_{\max}
C_u^*(\mS)$ and \cite[Theorem 5.12]{KPTT1} then clearly implies that
for the unital $C^*$-algebra $B = C^*_u(\mS)$, there does not exist a
$\ast$-isomorphism between $A \otimes_{C^*\text{-}\min} B$ and $ A
\otimes_{C^*\text{-}\max} B$ fixing $A \ot B$.
\end{proof}

Kirchberg had realized that exactness is quite a fundamental property
in the category of $C^*$-algebras and over the years it has shown its
prominence in the categories of operator spaces and operator systems
as well. The notion of exactness saw its relevance in the theory of
operator systems after Kavruk et al. developed an appropriate
formalism of the notion of quotient of operator systems in \cite[$\S
  3$]{KPTT2}:

Given an operator system $\mS$, a subspace $J\subseteq \mS$ is said to
be a \textit{kernel} (\cite[Definition $3.2$]{KPTT2}) if there exists
an operator system $\mT$ and a unital completely positive map $\phi :
\mS \ra \mT$ such that $J = \ker \phi$. For such a kernerl $J \subset
\mS$, Kavruk et al. showed that the quotient space $\mS/J$ forms an
operator system (\cite[Proposition $3.4$]{KPTT2}) with respect to the
natural involution, whose positive cones are given by \begin{eqnarray*}
  \mC_n &=& \mC_n(\mS/J)\\ &=& \{(s_{ij} + J)_{i,j} \in M_n(\mS/J):\;
  (s_{ij}) + \ep(1 + J)_n \in \D_n \;\text{for every} \; \ep >
  0\},\end{eqnarray*} where $$\D_n=\{(s_{ij}+J)_{i,j} \in M_n(\mS/J):
\;(s_{ij}+y_{ij})_{i,j} \in M_n(\mS)^+\; \text{for some $y_{ij} \in
  J$}\}$$ and the Archimedean unit is the coset $1 + J$.

For an operator system $\mS$,
a unital $C^*$-algebra $A$ and a closed ideal $I$ in $A$, let $\mS
\bar{\ot} I$ denote the closure of $\mS \ot I$ in the completion $\mS
\hat{\ot}_{\min} A$ of the minimal tensor product $\mS \ot_{\min}
A$. Then,  $\mS \bar{\ot} I$ is a
kernel in $\mS \hat{\ot}_{\mathrm{min}} A$ and the induced map $(\mS
\hat{\ot}_{\mathrm{min}} A)/(\mS \bar{\ot} I) \ra \mS
\hat{\ot}_{\mathrm{min}} (A/ I)$ is unital and completely positive (\cite[$\S 4$]{kavnuc}).

\begin{definition} \cite{KPTT2}\label{exactness} 
An operator system $\mS$ is said to be 
 $\mathrm{exact}$ if 
for every unital $C^*$-algebra $A$ and a closed ideal $I$ in $A$ the induced map
$$(\mS \hat{\ot}_{\mathrm{min}} A)/(\mS \bar{\ot} I) \ra \mS \hat{\ot}_{\mathrm{min}}
(A/ I)$$ is a complete order isomorphism.
 \end{definition}
Exactness is one of the few intrinsic properties of operator systems
that has been used as a tool, by Kavruk et al.~(see for example
\cite{kavnuc}), in characterizing nuclearity properties of operator
systems.  Recall the term \emph{nuclearity} for operator systems, a
generalization from the category of $C^*$-algebras, that was
introduced in \cite[$\S 3$]{KPTT1}:

 Given two operator system tensor
products $\alpha$ and $\beta$, an operator system $\mS$ is said to be
\emph{$(\alpha, \beta)$-nuclear} if the identity map between $\mS
\ot_{\alpha} \mT$ and $\mS \ot_{\beta} \mT$ is a complete order
isomorphism for every operator system $\mT$, i.e. $$\mS \ot_{\alpha}
\mT = \mS \ot_{\beta} \mT.$$ Also, an operator system $\mS$ is said to
be \textit{$C^*$-nuclear}, if $$\mS \ot_{\mathrm{min}} A =\mS
\ot_{\mathrm{max}} A$$ for all unital $C^*$-algebras $A$.

The following characterizations established in \cite[$\S 5$]{KPTT2}
and \cite[$\S 4$]{kavnuc} are used quite often:
\begin{theorem}\label{stability}
 \begin{enumerate}[(i)]
\item An operator system $\mS$ is exact if and only if it is $\mathrm{(min,el)}$-nuclear.
\item  Exactness passes to operator subsystems, i.e., if $\mS$ is exact then
every operator subsystem of $\mS$ is exact. Conversely, if every
finite dimensional operator subsystem of $\mS$ is exact then $\mS$ is
exact.
\item An operator system $\mS$ is $\mathrm{(min,c)}$-nuclear if and
  only if $\mS$ is $C^*$-nuclear.
\item An operator system is $\mathrm{(c,max)}$-nuclear if and only if
  it is unitally completely order isomorphic to a $C^*$-algebra.

\end{enumerate}
\end{theorem}

We now concentrate on the main class of operator systems that we study
in this article, namely, the operator systems associated to generating
sets of discrete groups.  Let $G$ denote a countable discrete group,
$\mF$ denote a generating set of $G$ and $\mS(\mF)$ denote the
operator system associated to $\mF$ by Farenick et al. in
\cite{disgrp}, i.e., $\mS(\mF) := \text{span}\{1, u, u^*:\ u \in \mF\}
\subset C^*(G)$, where $C^*(G)$ denotes the full group $C^*$-algebra
of the group $G$ (\cite[Chapter $8$]{pisier}).  It was shown in
\cite{disgrp} that if $\mF$ is a generating set of the free group
$\mathbb{F}_n$, then $\mS(\mF)$ is independent of the generating set
$\mF$ and is simply denoted by $\mS_n$.  In general, such independence
is not expected.

The following observation of \cite{disgrp} is immediate from
\Cref{unitary-envelope} and plays a fundamental role in the analysis
of nuclearity of operator systems associated to discrete groups.

\begin{proposition}\label{gpenvelope}
Let $\mF$ be a generating set of a discrete group $G$. Then, up to a
$*$-isomorphism that fixes the elements of $\mS(\mF)$, we have
$C^*_e(\mS(\mF))=C^*(G).$
\end{proposition}

Since the reduced group $C^*$-algebra is equally important as the full
group $C^*$-algebra, given any generating set $\mF$ of a discrete
group $G$, we associate another obvious operator system, namely,
$\mS_r(\mF) := \text{span}\{1, u, u^*:\ u \in \mF\} \subset
C^*_r(G)$. In view of \Cref{unitary-envelope}, analogous to
\Cref{gpenvelope} and the fact that $G$ is amenable if and only if
$C^*(G) = C^*_r(G)$ (\cite{lan, brownozawa}), we have:

\begin{proposition}\label{r-gpenvelope}
Let $\mF$ be a generating set of a discrete group $G$. Then,
\begin{enumerate}[(i)]
\item up to a $*$-isomorphism fixing the elements of $\mS_r(\mF)$, we
have $C^*_e(\mS_r(\mF)) = C^*_r(G)$.
\item $G$ is amenable if and only if the identity map on $\mF$
  extends to a complete order isomorphism between $\mS(\mF)$ and $
  \mS_r(\mF)$.
\end{enumerate}
\end{proposition}

Since finite dimensional $C^*$-algebras are nuclear and since
$C^*_r(\F_n)$ is exact (\cite{brownozawa}), by \Cref{stability},
\Cref{nuclear-C-cover}, \Cref{r-gpenvelope} and \Cref{unitary-envelope}, we observe that :

\begin{corollary}
\begin{enumerate}[(i)]
\item If $\mF$ is a generating set of a non-amenable discrete group
  $G$, then $\mS(\mF)$ and $\mS_r(\mF)$ do not possess nuclear
  $C^*$-covers.
 \item In particular, the finite dimensional operator systems $\mS(\mF)$ (e.g., $\mS_n$
   for $n \geq 2$) and $\mS_r(\mF)$, for $|\mF| < \infty$, do not admit 
   complete order embeddings into  finite dimensional $C^*$-algebras.
 \item If $\mF$ is a
 generating set of a free group with $2 \leq |\mF| \leq \infty$, then
 $\mS_r(\mF)$ is exact and does not have any nuclear $C^*$-cover.
\end{enumerate}
\end{corollary}

As in the case of group algebras, \Cref{gpenvelope} also suggests that
two non-isomorphic group operator systems can have isomorphic
$C^*$-envelopes: For example, consider the non-abelian groups $D_8$
(i.e., the Dihedral group of order $8$) with presentation $D_8 =
\langle a, b\; | \; a^2=b^4=1, bab=a \rangle$ and the Quaternion group
$Q_8$ with presentation $Q_8 =\langle x, y\; | \;x^4=1, x^2=y^2, xyx=y
\rangle $. It is well known that $$C^*(D_8)=\C \oplus \C \oplus \C
\oplus \C \oplus \M_2= C^*(Q_8)$$ and yet $D_8 \ncong Q_8$. Let $\mF=
\{a,b\}$ and $\mV=\{x,y\}$ be generating sets of $D_8$ and $Q_8$
respectively, then $\mF$ and $\mV$ are both minimal generating sets,
but $\mS(\mF) \ncong \mS(\mV)$ as $\dim(\mS(\mF))=4 \neq 5 =
\dim(\mS(\mV)).$

\section{A comparison between  $ess$ and other tensor products}\label{s:3}

Recall that, for operator systems $\mS$ and $\mathcal{T}$, analogous
to the commuting tensor product, their ess tensor product was
defined, in \cite{disgrp}, via the embedding $\mS \ot_{\mathrm{ess}}
\mathcal{T} \subset C^*_e(\mS) \ot_{\max} C^*_e(\mathcal{T})$. It was
also proved there, in Lemma $3.2$, that the operator systems
associated to free groups satisfy $\mS_n \ot_{\mathrm{ess}}
\mS_m=\mS_n \ot_{c} \mS_m $ for all $n, m \geq 2$. Analogous to this,
we will prove in this section that for operator systems associated
to amenable groups the $\mathrm{ess}$ tensor product is identical
with the maximal injective operator system tensor product e. Before
that, we first make some other useful observations about the tensor
product ess and compare it with other tensor products.

Analogous to the behavior of $\ot_{\mathrm{c}}$ in \cite[Theorem
  $6.6$]{KPTT1}, we have the following:
\begin{proposition} \label{ess-max}
For any two unital $C^*$-algebras $A$ and $B$, we have $$A \ot_{\mathrm{ess}} B
= A \ot_{\mathrm{c}} B= A \ot_{\max} B.$$
\end{proposition}
\begin{proof}
By \cite[Theorem $6.2.4$]{effros-ruan}, the injective envelope $I(A)$
has a canonical $C^*$-algebraic structure, and the mapping $i_{A}: A
\ra I(A)$ is a $C^*$-algebraic isomorphism onto its image. In
particular, we can assume that $A \subseteq I(A)$ and since the
$C^*$-envelope of $A$ is the $C^*$-algebra generated by $A$ in its
injective envelope $I(A)$, we have $C^*_e(A)=A$. Similarly
$C^*_e(B)=B$.  Hence, by the definition of $\ot_{\mathrm{ess}}$ and
the fact that $ A \ot_{c} B =A \ot_{\max} B$ (\cite[Theorem
  $6.6$]{KPTT1}), the assertion holds.
\end{proof}

Recall, from \cite[$\S 3$]{KPTT1}, that for two operator system tensor
products $\sigma$ and $\tau$, one says that $\sigma \leq \tau$ if for
any two operator systems $\mS$ and $\mT$ the identity map from $\mS
\ot_{\tau}\mT$ onto $\mS \ot_{\sigma}\mT$ is completely positive. With
this notion, the following lattice structure on operator
system tensor products is known (\cite{KPTT1, KPTT2, kavnuc}):
$$ \min \leq e \leq \mathrm{el}, \mathrm{er} \leq \mathrm{c} \leq \max.$$
Also, it can be easily seen that for three operator system tensor
products $\sigma \leq \tau \leq \rho$, an operator system $\mS$ is
$(\sigma, \rho)$-nuclear if and only if it is $(\sigma, \tau)$- and
$(\tau, \rho)$-nuclear.

Further, an operator system tensor product $\alpha$ is said to be
\emph{functorial} if for operator systems $\mS_1, \mS_2, \mT_1 \:
\text{and} \: \mT_2$ and unital completely positive maps $\phi : \mS_1
\ra \mS_2$ and $\psi : \mT_1 \ra \mT_2$ the associated map $\phi \ot
\psi : \mS_1 \ot_{\alpha} \mT_1 \ra \mS_2 \ot_{\alpha} \mT_2$ is
unital completely positive (\cite{KPTT1}). A tensor product $\alpha$
is said to be \emph{symmetric} if the flip map $\theta : s \ot t \mapsto t
\ot s$ extends to unital complete order isomorphism from $\mS
\ot_{\alpha} \mT$ onto $\mT \ot_{\alpha} \mS$, and \emph{associative}
if the natural isomorphism from $(\mS \ot \mT ) \ot \mathcal{R} $ onto
$\mS \ot (\mT \ot \mathcal{R})$ yields a complete order isomorphism
from $(\mS \ot_{\alpha} \mT ) \ot_{\alpha} \mathcal{R}$ onto $ \mS
\ot_{\alpha} (\mT \ot_{\alpha} \mathcal{R} )$ for all operator systems
$\mS, \mT$ and $\mathcal{R}$.

\begin{proposition}
$\mathrm{ess}$ is symmetric and is not functorial .
\end{proposition}
\begin{proof} 
Since $\max$ is symmetric (\cite[Theorem $5.5$]{KPTT1}), the flip map 
$$ C^*_e(\mS) \ot_{\max} C^*_e(\mT) \ni a\ot b \stackrel{\Phi}{
  \longrightarrow} b \ot a \in C^*_e(\mT) \ot_{\max} C^*_e(\mS)$$ extends to a
complete order isomorphism for any two operator systems $\mS$ and
$\mT$ and, by definition of $\ot_{\mathrm{ess}}$, the restriction of
$\Phi$ to $\mS \ot \mT$ implies the symmetry for $\ot_{\mathrm{ess}}$.

Then, in view of \Cref{ess-max}, the facts that $\mathrm{c}$ is the
minimal symmetric and functorial extension of max which agrees with
max on $C^*$-algebras (\cite[Theorem $6.7$]{KPTT1}) and that
$\mathrm{ess} \leq \mathrm{c}$ (\cite[$\S 2$]{disgrp}) imply that ess
is not functorial. \end{proof} 

Like the tensor product ${\text{c}}$ it is not known whether
${\text{ess}}$ is associative or not. As applications of our main
results, we will be able to make some more significant comparisons
between ess and other tensor products in the following Section.

\vspace*{1mm} We now review some basic facts about dual of an operator
system from \cite{choi, KPTT2}: \\ For an operator system $\mS$, the
Banach space dual $\mS^d$ is a matrix ordered space with
ordering: $$M_n(\mS^d) \ni (f_{ij}) \geq 0 \ \ \text{if and only
  if} $$ $$F: \mS \ra M_n\; \text{given by}\; F(s)=(f_{ij}(s))\;
\text{is completely positive.}$$ But $\mS^d$ need not be an operator
system, as it might not have an Archimedean ordered unit. However, if
$\mS$ is finite dimensional, the dual $\mS^d$ possesses an Archimedean
order unit and hence admits an operator system structure
(\cite[Corollary $4.5$]{choi}).  It was shown in \cite[Proposition
  6.2]{KPTT2} that $\mS^{dd}$ is always an operator system with a
canonical Archimedean ordered unit and the canonical inclusion $\mS
\subset \mS^{dd}$ is a complete order embedding.

As in (\cite[Definition $6.4$]{KPTT2}), an operator system $\mS$
is said to have \textit{the Weak Expectation Property (WEP)} if
there exits a complete order embedding $\mS \subset B(H)$ such that the canonical
map $\iota : \mS \ra \mS^{dd}$ extends to a completely positive map
$\tilde{\iota} : B(H) \ra \mS^{dd}$. It is known that $\mS$ has the
WEP if and only if it is $(\mathrm{el}, \max)$-nuclear (\cite[$\S 6$]{KPTT2} , \cite[$\S 4$]{han}).

\begin{proposition}\label{ess-e1}
Let  $\mS$ and $\mT$ be operator systems whose $C^*$-envelopes possess
the WEP. Then, $\mS \ot_{\mathrm{ess}} \mT=\mS \ot_{\mathrm{e}} \mT$.
\end{proposition}
\begin{proof} 
By \cite[Theorem $6.9$]{KPTT2} (also see \cite[Corollary $4.2$]{han1}
and \cite[Corollary $3.6.8$]{brownozawa}), a unital $C^*$-algebra $A$
possesses the WEP if and only if $A \ot_{\max} B \subseteq A_1
\ot_{\max} B$ for any inclusion $A \subseteq A_1$ and any unital
$C^*$-algebra $B$ (Lance's characterization
of WEP). Thus,
$$ \mS \ot_{\mathrm{ess}} \mT  \subseteq  C^*_e(\mS) \ot_{\max} C^*_e(\mT)
  \subseteq  I(\mS) \ot_{\max} I(\mT) 
$$ and consequently $\mS \ot_{\mathrm{ess}} \mT=\mS \ot_{\mathrm{e}} \mT$.
\end{proof}

This allows us to deduce what we had promised at the beginning of this
section, i.e.,  analogous to \cite[Lemma $3.2$]{disgrp}, for amenable groups,
we have:

\begin{corollary} \label{ess-e}
Let $\mF$ and $\mV$ be generating sets of amenable discrete groups $G$
and $H$ respectively. Then, $\mS(\mF)
\ot_{\mathrm{ess}} \mS(\mV) = \mS(\mF) \ot_{\mathrm{e}} \mS(\mV)$.
\end{corollary}
\begin{proof}
Since the full group $C^*$-algebras of amenable groups are nuclear,
and the fact that every nuclear $C^*$-algebra has the WEP (\cite[$\S
  17$]{pisier}), by \Cref{gpenvelope} and \Cref{ess-e1} proof follows.
\end{proof}
In view of \Cref{r-gpenvelope}, \Cref{ess-e} holds for $\mS_r(\mF)$
and $\mS_r(\mV)$ as well.  In \cite{KPTT2}, a generalization of the
notion of WEP was introduced and was called \textit{the Double
  Commutant Expectation Property (DCEP)}. An operator system $\mS$ is
said to have the DCEP if for every complete order embedding $\mS
\subset B(H)$ there exists a completely positive map $\varphi : B(H)
\ra \mS''$ fixing $\mS$.
\begin{proposition} 
Let $A$ be a unital $C^*$-algebra, then for every operator system
$\mS$ possessing the DCEP, we have $$\mS \ot_{\mathrm{ess}} A =\mS \ot_{\mathrm{el}}
A = \mS \ot_{\max} A.$$
\end{proposition} 

\begin{proof} 
By \cite[Theorems $7.1$ and $7.3$]{KPTT2}, an operator system $\mS$
has the DCEP if and only if for every operator system $\mathcal{S}_1$
with $\mS \subseteq \mS_1$ and any operator system $\mathcal{R}$ we
have $\mS \ot_{\mathrm{c}} \mathcal{R} \subseteq_{\mathrm{coe}} \mT
\ot_{\mathrm{c}} \mathcal{R}$.  In particular, by \cite[Theorem
  $6.7$]{KPTT1} and the fact that $A = C^*_e(A)$ (as observed in
\Cref{ess-max}), we have $$\mS \ot_{\max} A = \mS \ot_{\mathrm{c}} A
\subset C^*_e(\mS) \ot_{\mathrm{c}} A = C^*_e(\mS) \ot_{\max} A
\supset \mS \ot_{\mathrm{ess}} A.$$ Also, by \cite[Theorem
  $7.3$]{KPTT2} again, $\mS$ has the DCEP if and only if it is
$(\mathrm{el}, \mathrm{c})$-nuclear and we are done.
\end{proof}

\section{Nuclearity of an operator system via its $C^*$-envelope}\label{s:4}
 In this section we compare various notions of nuclearity  of operator
 systems with their $C^*$-envelopes.
 \vspace*{1mm}
 
Kirchberg and Wassermann (\cite{kirchberg}) gave
an example of a $(\min, \max)$-nuclear operator system whose
$C^*$-envelope, as observed by Kavruk in \cite[$\S 6$]{kavnuc}, is
non-exact and hence non-nuclear. The other direction, in general, is equally mysterious.

\begin{proposition}\label{notnuclear1} $(\min,
\mathrm{c})$-nuclearity, $(\min, \mathrm{er})$-nuclearity and
$(\mathrm{el}, \mathrm{c})$-nuclearity do not pass to an operator system
from its $C^*$-envelope.
\end{proposition}
\begin{proof}
For the operator system $\mS_2$ associated to the free group with two
generators, there exists a complete order embedding of its dual
$\mS_2^d$ into $M_4$ (\cite[Theorem 4.4]{opsysqt}, \cite[Theorem
  10.11]{kavnuc}). And, since $\mS_2$ is not exact, by
\cite[Corollary $10.14$]{kavnuc}, $\mS_2^d$ is not $(\min,
\mathrm{er})$-nuclear.  Hence, it fails to be $(\min,
\mathrm{c})$-nuclear as $\mathrm{er} \leq \mathrm{c}$. In
particular, as exactness passes to operator subsystems
(\Cref{stability}), $\mS_2^d \subseteq M_4$ is exact and hence
$(\min,\mathrm{el})$-nuclear \cite[$\S 5$]{KPTT2}. Therefore, it is
not $(\mathrm{el}, \mathrm{c})$-nuclear as well. 

By \Cref{nuclear-C-cover}, $C^*_e(\mS_2^d)$ is nuclear. Thus, neither of
$(\min, \mathrm{c})$-nuclearity, $(\min, \mathrm{er})$-nuclearity or
$(\mathrm{el}, \mathrm{c})$-nuclearity passes to an operator system
from its $C^*$-envelope.
\end{proof} 

However, by the very way $\ot_{\mathrm{ess}}$ is defined, we have the following:

\begin{proposition}\label{mainresult} 
An operator system is $(\min, \mathrm{ess})$-nuclear if its
$C^*$-envelope is nuclear. Moreover, a unital $C^*$-algebra is $(\min,
\mathrm{ess})$-nuclear as an operator system if and only if it is
nuclear as a $C^*$-algebra.
\end{proposition}
\begin{proof}
Let $\mS$ be an operator system with nuclear $C^*$-envelope. By
injectivity of $\ot_{\min}$ we have $ \mS \ot_{\min} \mT \subseteq
C^*_e(\mS) \ot_{\min} C^*_e(\mT)$ and, by the definition of
$\ot_{\mathrm{ess}}$, $\mS \ot_{\mathrm{ess}} \mT \subseteq C^*_e(\mS)
\ot_{\max} C^*_e(\mT)$ for any operator system $\mT$. By  \cite[Corollary
  $6.8$]{KPTT1} a nuclear $C^*$-algebra is $(\min,\max)$-nuclear as an operator system ,
$C^*_e(\mS) \ot_{\min} C^*_e(\mT) = C^*_e(\mS) \ot_{\max} C^*_e(\mT)$
and, hence, $\mS$ is $(\min, \mathrm{ess})$ nuclear.

 Conversely, if $A$ is a unital
$C^*$-algebra which is $(\min, \mathrm{ess})$-nuclear then, by
\Cref{ess-max}, $A \ot_{\min} B =  A \ot_{\mathrm{ess}} B = A \ot_{\max} B $ for every
unital $C^*$-algebra $B$. Therefore,  by
\cite[Proposition $4.11$]{kavnuc}, $A$ is $(\min, \mathrm{c})$-nuclear
as an operator system and hence, by \cite[Theorem $6.7$ and Corollary
  $6.8$]{KPTT1}, $A$ is nuclear as a $C^*$-algebra.
 \end{proof}

We do not know whether the $C^*$-envelope of a $(\min,
\mathrm{ess})$-nuclear operator system is nuclear in general or
not. However, the situation is better in the enough unitaries case:

\begin{theorem}\label{mainresult2}
Let $\mS \subset A$ contain enough unitaries of the unital
$C^*$-algebra $A$. Then $\mS$ is $(\min, \mathrm{ess})$-nuclear if and
only if $A$ is a nuclear $C^*$-algebra.
\end{theorem}

\begin{proof}
If $A$ is nuclear, then, it follows from \Cref{unitary-envelope} and
\Cref{mainresult} that $\mS$ is $(\min, \text{ess})$-nuclear.
Conversely, suppose $A$ is not nuclear. Then, by \Cref{non-nuclear},
there exists a unital $C^*$-algebra $B$ such that the identity map on
$A \ot B$ does not extend to a $\ast$-isomorphism between $A
\otimes_{C^*\text{-}\min} B$ and $A \otimes_{C^*\text{-}\max} B$. Note
that, by injectivity of $\ot_{\min}$, $\mS \otimes_{\min} B $ has
enough unitaries in $A \otimes_{C^*\text{-}\min} B$ and by definition
of $\ot_{\mathrm{ess}}$, so does $\mS \otimes_{\mathrm{ess}} B $ in $A
\otimes_{C^*\text{-}\max} B$.  Thus, by \Cref{unitary-envelope} again,
$A \otimes_{C*\text{-}\min} B$ is the $C^*$-envelope of $\mS
\otimes_{\min} B $ and likewise $A \otimes_{C^*\text{-}\max} B$ is
that of $\mS \otimes_{\mathrm{ess}} B$. In particular, this implies
that $\mS \otimes_{\min} A \neq \mS \otimes_{\mathrm{ess}} A$ and
hence $\mS$ is not $(\min, \text{ess})$-nuclear.
\end{proof}

Recall from \cite[$\S
2$]{disgrp} that $\mathrm{ess} \leq \mathrm{c}$. However, as an
  application of our main results \Cref{mainresult} and
  \Cref{mainresult2} the next proposition shows that $\mathrm{ess}$
  does not compare that well with er and el.

\begin{proposition}\label{comparison}
$\mathrm{er} \nleq \mathrm{ess}$, $\mathrm{ess} \nleq \mathrm{er}$ and
  $ \mathrm{ess} \nleq \mathrm{el}$.
\end{proposition}
\begin{proof} 
 
We saw in \Cref{notnuclear1} that there exists
a complete order embedding of  $\mS_2^d$ into $M_4$. So, by \Cref{mainresult}, $\mS_2^d $ is $(\min,
\mathrm{ess})$-nuclear.  In \Cref{notnuclear1}, we also saw that $\mS_2^d \subset M_4$ is not
$(\min,\mathrm{er})$-nuclear. This implies that  $\mathrm{er} \nleq
\mathrm{ess}$.

 Then, by \cite[Proposition $9.9$]{KPTT2}, $\mS_2$ is
$(\min, \mathrm{er})$-nuclear but, by \Cref{mainresult2}, it is not
$(\min, \mathrm{ess})$-nuclear and hence $\mathrm{ess} \nleq
\mathrm{er}$. Finally, if $\mF$ is a generating set of a free group
$\mathbb{F}$ with $2 \leq |\mF| \leq \infty$, then $C^*_r(\mathbb{F})$
being exact (\cite[Chapter 8]{pisier},\cite[Proposition 5.1.8]{brownozawa}), $\mS_r(\mF)$ is $(\min,
\mathrm{el})$-nuclear by \Cref{stability}; and on the other hand, by
\Cref{mainresult2} again, $\mS_r(\mF)$ is not $(\min,
\mathrm{ess})$-nuclear implying that $\mathrm{ess} \nleq \mathrm{el}$. \end{proof}

It is not clear whether $\mathrm{el} \leq \mathrm{ess} $ or
not. However, we will prove later (in \Cref{ess-to-el}) that $(\min,
\mathrm{ess})$-nuclearity implies $(\min,
\mathrm{el})$-nuclearity.

As an immediate consequence of \Cref{mainresult}, we observe that $(\min,
\mathrm{ess})$-nuclearity passes to an operator system from its
$C^*$-envelope. However, since a unital $C^*$-algebra is $(\min,
\mathrm{ess})$-nuclear if and only if it is nuclear, and as nuclearity
is not preserved by $C^*$-subalgebras (\cite[$\S 4$]{brownozawa}), we
observe that, in general:

\begin{remarks}  $(\min, \mathrm{ess})$-nuclearity does not
pass to operator subsystems.
\end{remarks}
      
\begin{proposition} 
 $(\mathrm{el}, \mathrm{c})$-nuclearity and $(\mathrm{el},
  \max)$-nuclearity do not pass to operator subsystems.
  \end{proposition}

   \begin{proof} Since a $C^*$-algebra is nuclear if and only if it is exact
  and has the $\mathrm{WEP}$ (\cite[$\S 17$]{pisier}), and as
  exactness passes to $C^*$-subalgebras, the failure of passage of
  nuclearity to $C^*$-subalgebras \cite{brownozawa} also shows that
  the $\mathrm{WEP}$ does not pass to $C^*$-subalgebras. And as the
  $\mathrm{WEP}$ and the $\mathrm{DCEP}$ are the same for a
  $C^*$-algebra (\cite{KPTT2, kavnuc}), this also illustrates that the
  DCEP (equivalently, $(\mathrm{el}, \mathrm{c})$-nuclearity) and the
  WEP (equivalently, $(\mathrm{el}, \max)$-nuclearity) do not pass to
  operator subsystems in general.
\end{proof}

On the other hand, following is an
immediate consequence of \Cref{nuclear-C-cover} and \Cref{mainresult}:

\begin{corollary}
 An operator system  $\mS$ is $\mathrm{(min, ess)}$-nuclear if it admits a nuclear 
$C^*$-cover. \label{nuclear-cover} 

\end{corollary}

In particular, we also see  that if $\mS \subset A$ is an operator
subsystem of a finite dimensional $C^*$-algebra $A$ then $\mS$ is
$\mathrm{(min, ess)}$-nuclear. We saw in \Cref{notnuclear1} that $\mS_2^d \subset M_4$ is not $(\min,
\mathrm{c})$-nuclear; so, it also serves as an example of a
finite dimensional (exact) operator system which is not
$(\mathrm{ess}, \mathrm{c})$-nuclear. We have thus observed that:

\begin{remarks}
An operator system need not be $(\mathrm{ess}, \mathrm{c})$-nuclear if
it possesses a nuclear $C^*$-envelope. In particular, $(\mathrm{ess},
\mathrm{c})$-nuclearity does not pass to operator subsystems.
\end{remarks}

By \Cref{mainresult}, every $C^*$-algebra which is $(\min,
\text{ess})$-nuclear is also nuclear and hence exact.  In fact, the
same is true for operator systems as well, which will follow from the
proposition given below, where the notations $\hat{\ot}$ and $\bar{\ot}$ have
similar meanings as in \Cref{exactness}:

\begin{proposition}\label{ess-el}
Let $\mS$ be an operator system and $\mathcal{I}$ be a closed ideal in a unital
$C^*$-algebra $A$. Then $\mS \bar{\ot} \mathcal{I}$ is a completely
biproximinal kernel in $\mS \hat{\ot}_{\mathrm{ess}} A$ and the induced
map $$\frac{\mS \hat{\ot}_{\mathrm{ess}} A}{\mS \bar{\ot} \mathcal{I}}
\ra \mS \hat{\ot}_{\mathrm{ess}} A / \mathcal{I}$$ is a unital
complete order isomorphism.
\end{proposition}
\begin{proof}
Since ess is induced by a $C^*$-algebraic tensor product, by
\cite[Proposition $5.14$]{KPTT2}, $\mS \bar{\ot} \mathcal{I}$ is a
completely biproximinal kernel in
$\mS \hat{\ot}_{\mathrm{ess}} A$, i.e., $ \mcal{C}_n(\frac{\mS
  \hat{\ot}_{\mathrm{ess}} A}{\mS \bar{\ot} \mathcal{I}}) =
\mcal{D}_n(\frac{\mS \hat{\ot}_{\mathrm{ess}} A}{\mS \bar{\ot}
  \mathcal{I}})$ for all $n\geq1$  (\cite[Definition $4.9$]{KPTT2}), and the induced map $\varphi:
\frac{\mS \hat{\ot}_{\mathrm{ess}} A}{\mS \bar{\ot} \mathcal{I}} \ra
\frac{C^*_e(\mS) \hat{\ot}_{\max} A}{C^*_e(\mS) \hat{\ot}_{\max}
  \mathcal{I}}$ is a complete order isomorphic inclusion. On the other
hand, the canonical map $\theta: \frac{C^*_e(\mS) \hat{\ot}_{\max}
  A}{C^*_e(\mS) \hat{\ot}_{\max} \mathcal{I}} \ra C^*_e(\mS)
\hat{\ot}_{\max} A/ \mathcal{I}$ is a complete order isomorphism (by
\cite[Corollary $15.6$]{KPTT2}) and we have $ \mS
\hat{\ot}_{\mathrm{ess}} A / \mathcal{I} \subseteq C^*_e(\mS)
\hat{\ot}_{\max} A/ \mathcal{I}$. Clearly, $\theta \circ \varphi$ is a
surjection onto $\mS \hat{\ot}_{ess} A / \mathcal{I}$ and agrees with
the induced map $\frac{\mS \hat{\ot}_{\mathrm{ess}} A}{\mS \bar{\ot}
  \mathcal{I}} \ra \mS \hat{\ot}_{\mathrm{ess}} A / \mathcal{I}$ and
hence the assertion holds.

\end{proof}

\begin{corollary}\label{ess-to-el}
Every operator system which is $(\min, \mathrm{ess})$-nuclear is also
$(\min, \mathrm{el})$-nuclear.
\end{corollary}

As remarked earlier, we don't know whether $\mathrm{el} \leq \mathrm{ess}$ or not but
\Cref{ess-to-el} does hint that it could very well be true.

The converse to \Cref{ess-to-el} is false. Consider the free group
$\F_n$ with $n$ generators and $n\geq 2$. Its reduced group
$C^*$-algebra $C_r^*(\F_n)$ is exact, i.e., $(\min,
\mathrm{el})$-nuclear (\cite{brownozawa, pisier}) but not
$C^*$-nuclear and hence, by \Cref{mainresult}, it is not $(\min,
\mathrm{ess})$-nuclear.
\vspace*{1mm}

 Before moving on to the section on operator systems associated to
 discrete groups, we point out that, unlike the category of
 $C^*$-algebras, exactness, (min, ess)-nuclearity, (min, c)-nuclearity and (min,
 max)-nuclearity do not pass to operator system quotients.
 
\begin{corollary}
$(\min, \mathrm{ess})$-nuclearity does not pass to operator system quotients.
\end{corollary}
\begin{proof}
For each $n \geq 2$, let ${J}_n \subset {M}_n(\C)$ be the kernel
(\cite[$\S 2$]{opsysqt}) in ${M}_n(\C)$ consisting of all diagonal
matrices $D \in {M}_n(\C)$ with $\mathrm{tr}(D)=0$; and,
$\mathcal{W}_n$ be the operator subsystem of $C^*(\mathbb{F}_{n-1})$
spanned by $\{u_iu_j^* : 1 \leq i,j \leq n\}$ where $u_2,\ldots, u_n$
are the universal unitaries that generate $C^*(\F_{n-1})$ and
$u_1:=1$. For $n \geq 3$, by \cite[Theorem 2.4]{opsysqt}, ${M}_n/J_n$
is completely order isomorphic to $\mathcal{W}_n$ and $
C^*_e(\mathcal{W}_n) = C^*_e({M}_n/J_n) =C^*(\F_{n-1})$. By
\Cref{mainresult2}, we deduce that the quotient ${M}_n/J_n$ is not $(\min,
\mathrm{ess})$-nuclear.
\end{proof}

 In addition, the example of $M_n/J_n$ also shows that
 $C^*_e({M}_n/J_n)$ is isomorphic to $C^*(\F_{n-1})$, which is not exact
 (\cite[$\S 17 $]{pisier}).  Now, since $\mathcal{W}_n$ contains
 enough unitaries of $C^*(\mathbb{F}_{n-1})$, $\mathcal{W}_n$ and
 hence ${M}_n/J_n$ is not exact (by \cite[Corollary $9.6$]{KPTT2}). In
 particular, exactness, $\mathrm{(min, c)}$-nuclearity, and
 $\mathrm{(min, max)}$-nuclearity do not pass to quotients in the
 category of operator systems as well.
\vspace*{1mm}

 Note that, Kavruk (\cite[Theorem $10.2$]{kavnuc}) showed that the
 quotient operator system $M_3/J_3$ has a deep relationship with
 Kirchberg's conjecture. It was established that Kirchberg's
 conjecture has a positive answer if and only if $M_3/J_3$ possesses
 the DCEP. There is one more quotient operator system which is equally
 deeply related to Kirchberg's conjecture, namely, $T_n /J_n$
 where $T_n:=\{[a_{i,j}] \in M_n : a_{i, j} = 0 \ \text{if}\ | i - j|
 > 1 \}$. It was also established, in \cite[Corollary
   $10.6$]{kavnuc}, that  Kirchberg's conjecture is true if and
 only if $T_3 /J_3$ possesses the DCEP. This was a consequence of the
 fact \cite[Theorem $10.5$]{kavnuc} that $T_n /J_n$ is completely
 order isomorphic to $\mS_{n-1}$ for all $n \geq 3$. We easily deduce,
 again from \Cref{mainresult2}, that the quotient operator system $T_3
 /J_3$ is not $(\min, \mathrm{ess})$-nuclear.
\vspace*{1mm}

Furthermore, as observed in \Cref{comparison}, the dual $\mS_2^d$ of the free group operator
system $\mS_2$  is
(min, ess)-nuclear, whereas $\mS_2^{dd}= \mS_2$ (upto complete order
isomorphism) is not (min, ess)-nuclear and hence we also observe that:
\begin{corollary}
$(\min, \mathrm{ess})$-nuclearity does not pass to operator system  duals.
\end{corollary}

\section{Nuclearity of group operator systems}\label{s:5}

Recall that if $\mS(\mF)$ is a $(\min, \mathrm{c})$-nuclear operator
system then $C^*(G)$ is a nuclear $C^*$-algebra (\cite[Corollary
  $9.6$]{KPTT2}) and since the group $C^*$-algebra of a discrete group
is nuclear if and only if the group is amenable (\cite{lan,
  brownozawa}), it follows that the group is amenable. By
\Cref{mainresult2}, we realize that amenability can, in fact, be
recovered from $(\min, \mathrm{ess})$-nuclearity itself. Suprisingly, it
is not yet clear (at least, to us) whether the amenability of a group
guarantees $(\min,\mathrm{ c})$-nuclearity of an operator system
associated to a generating set of the group. The following can be
treated as a small step in this direction, whose proof is now an
immediate consequence of \Cref{gpenvelope} and \Cref{mainresult2}:

\begin{theorem}\label{gp-nuclearity}
Let $\mF$ be a generating set of a discrete group $G$. Then the operator systems
$\mS(\mF)$ and $\mS_r(\mF)$ are $(\min, \mathrm{ess})$ nuclear if and
only if $G$ is amenable.
\end{theorem}

Note that \Cref{gp-nuclearity}, kind of, falls in line with the
fact that amenability of a discrete group is equivalent to the
nuclearity of its reduced group $C^*$-algebra $C^*_r(G)$ and  injectivity of its group von
Neumann algebra $L(G)$ (\cite[$\S 3$]{brownozawa}, \cite{lan}).

We now move towards identifying the $(\min, \max)$-nuclear group
opertor systems associated to finitely generated groups. For a finite
dimensional operator system $\mS$, it is known that $\mS$ is
$\mathrm{(c,max)}$-nuclear if and only if it is unitally completely
order isomorphic to a $C^*$-algebra (\cite[ Proposition 4.12]
{kavnuc}). In fact, since $\mathrm{ess} \leq \mathrm{c}$, this allows
us to conclude the following:

  \begin{proposition}\label{fd} 
Let $\mS$ be a  finite dimensional operator system. Then the following are equivalent:
\begin{enumerate}[(i)] 
\item $\mS$ is $(\mathrm{ess}, \max)$-nuclear.
\item $\mS$ is $(\mathrm{c}, \max)$-nuclear.
\item $\mS$ is unitally completely order isomorphic to a $C^*$-algebra.
\item $\mS$ is $(\min, \max)$-nuclear.
\end{enumerate}
\end{proposition}

Since an injective operator system admits the structure of a $C^*$-algebra
(\cite[Theorem 3.1]{choi}), \Cref{fd} has the following
immediate consequence:

\begin{corollary}\label{injective}
 Let $\mF$ be a finite generating set of a discrete group $G$. Then
 $\mS(\mF)$ is $\mathrm{(min,max)}$-nuclear if and only if $\mS(\mF)$
 is injective.
\end{corollary}

We now provide the promised exhaustive list of nuclear group operator
systems associated to minimal generating sets of finitely generated
groups.
\begin{theorem}\label{classfn}
 Let $\mF$ be a  generating set of a finitely generated  discrete group $G$. 
\begin{enumerate}[(i)]
\item If $\mF$  is finite and $\mS(\mF)$  is $\mathrm{(min,max)}$-nuclear, then $G$ is
  finite. 
\item If $\mF$ is a minimal generating set, then $\mS(\mF)$ is
  $\mathrm{(min, max)}$-nuclear if and only if $G$ is of order less
  than or equal to $3$.
\end{enumerate}
\end{theorem}

\begin{proof}
We prove both assertions simultaneously. Note that since $G$ is
finitely generated, $\mF$ is finite in $(ii)$ as well.  So, by
\Cref{fd}, $\mS(\mF)$ is (min,max)-nuclear if and only if $\mS(\mF)$
is completely order isomorphic to a $C^*$-algebra. Then, by
\Cref{rigidity}$(iii)$, $\mS(\mF)$ is completely order isomorphic to a
$C^*$-algebra if and only if $\mS(\mF)=C^*_e(\mS(\mF))=C^*(G)$, which
is true if and only if $\mF \cup \mF^{-1} \cup \{ e\} = G$, where
$\mF^{-1}:= \{ u^{-1} : u \in \mF\}$.

Suppose $\mS(\mF)$ is $(\min, \max)$-nuclear. Since $\mF$ is finite,
$\mS(\mF) = C^*(G)$ is finite dimensional and hence $G$ is
finite. Note that if $|\mF| > 1$ and $a, b \in \mF$ with $a \neq b$
then, by minimality of $\mF$, the set $\{ ab\} \cup \mF \cup \mF^{-1}
\cup \{e\}$ is linearly independent in $C^*(G)$ and therefore $ab
\notin \mS(\mF)$. So, $\mF$ must be a singleton, i.e., $G$ must be
cyclic. And clearly, for a singleton $\mF$, the equality $\mF \cup \mF^{-1}
\cup \{ e\} = G$ holds only if $G = \Z_1, \Z_2$ or $\Z_3$. Conversely,
if $G = \Z_1, \Z_2$ or $\Z_3$, then clearly $\mF \cup \mF^{-1} \cup \{
e\} = G$ and $\mS(\mF)$ is then $(\min, \max)$-nuclear.
\end{proof}

Note that minimality of $\mF$ can not be dropped from the statement of
\Cref{classfn}. For example, if $G = \Z_3 \oplus \Z_3$ and $\mF = \{
(1, 0), (1, 1), (0,1), (1, 2)\}$, then $\mS(\mF)$ is $(\min,
\max)$-nuclear but $G$ is neither cyclic nor $|G| \leq 3$. On the
other hand, it is not yet clear whether $\mS(\mF)$ is $(\min,
\max)$-nuclear if $\mF$ is infinite even if it equals $G$ or $G
\setminus \{ e\}$.

We thus obtain yet another collection of finite dimensional (min,
ess)-nuclear operator systems which are not (min, max)-nuclear.
\begin{corollary}
For any finitely generated amenable group $G$ with $|G| \geq 4$ and
any finite generating set $\mF$ of $G$ such that $\mF \cup \mF^{-1}
\cup \{ e\} \neq G$,
$\mS(\mF)$ is $\mathrm{(min, ess)}$-nuclear but not $\mathrm{(min,
  max)}$-nuclear. In particular, $\mathrm{(ess, max)}$-nuclearity does
not pass to operator subsystems.
\end{corollary}

The preceding Corollary also gives examples of operator systems which
are not (min, max)-nuclear and yet possess nuclear $C^*$-envelopes.


\section{Nuclearity of graph operator systems}\label{s:6}

Given a finite graph $G$ with $n$-vertices, Kavruk et al.~in
\cite{KPTT1} associated an operator system $\mS_G$ as the finite
dimensional operator subsystem of $M_n(\C)$ given by
 $$\mS_G=\mathrm{span} \{\{E_{i,j}: (i,j) \in G\}\cup\{E_{i,
  i}: 1 \leq i \leq n\} \} \subseteq M_n(\C),$$ where $\{E_{i,j}\}$ is
the standard system of matrix units in $M_n(\C)$ and $(i, j)$ denotes
(an unordered) edge in $G$.  In view of \Cref{mainresult}, the
graph operator system $\mS_G\subseteq M_n$ is always $(\min,
\mathrm{ess})$-nuclear.

It is known that if $G$ is chordal, i.e., no minimal cycle of $G$ has
length greater than $3$, then $\mS_G$ is $C^*$-nuclear
(\cite[Proposition 6.11]{KPTT1}). It is not known whether the converse
is true or not. However, motivated by the discussions in \cite[$\S
  3$]{OP}, we obtain the following characterization of the $(\min,
\max)$-nuclear graph operator systems.

\begin{theorem}
Let $G$ be a finite graph. Then the associated operator system
$\mS_G$ is $(\min, \max)$-nuclear if and only if each 
component of $G$ is  complete.
\end{theorem}
\begin{proof}
Suppose each connected component of $G$ is a complete
graph. Then we easily see that $\mS_G= M_{n_1} \oplus \cdots
\oplus M_{n_k} \subset M_{n}$, where $n_i$'s are the number of vertices
in the connected components of $G$, $k$ is the number of
connected components of $G$ and $n$ is the number of vertices
in $G$. Thus, $\mS_G$ is $(\min, \max)$-nuclear.

Conversely, suppose $\mS_G$ is $(\min, \max)$-nuclear.  Then, by
\Cref{rigidity} $(iii)$, we have $\mS_G= C^*_e(\mS_G)$ and, Ortiz and
Paulsen proved in \cite[Theorem 3.2]{OP} that $C^*_e(\mS_G) =
C^*(\mS_G) \subseteq M_n(\C)$ for any graph $G$. Let $\mcal{F}$ be a connected component
of $G$ and $v \neq w$ be any two vertices in $\mcal{F}$. Then there
exists a sequence of connected edges $(v, i_1), (i_1, i_2), (i_2, i_3),$ $
\ldots, (i_r, w)$ in $\mcal{F}$ connecting $v$ with $w$. Further,
since $\mS_G= C^*(\mS_G)$, we get $E_{v, w} = E_{v, i_1} E_{ i_1, i_2}
\cdots E_{i_{r-1}, i_r} E_{i_r, w} \in \mS_G$. This implies that $(v,
w) \in G$, i.e., $v$ and $w$ are connected by an edge and hence
$\mcal{F}$ is complete.
\end{proof}

\section{Nuclearity properties of some known examples}\label{s:7}

\subsection{Operator systems of Commuting and Non-Commuting n-cubes} 
Inspired by Tsirelson's non-commutative analogues of n-dimensional
cubes, Farenick et al. (in \cite{disgrp}) introduced an $(n+1)$-dimensional operator
system  $NC(n)$ as follows:
\vspace*{1mm}

 Let $\mathcal{G} = \{h_1, . . . , h_n\}$, let
$\mathcal{R} = \{h^*_j = h_j , \|h\| \leq 1, 1 \leq j \leq n\}$ be a
set of relations in the set $\mathcal{G}$, and let 
$C^*(\mathcal{G}|\mathcal{R})$ denote the universal unital $C^*$-algebra
generated by $\mathcal{G}$ subject to the relations $\mathcal{R}$. The operator
system $$NC(n) := span\{1, h_1, ..., h_n\} \subset C^*(\mathcal{G}|\mathcal{R})$$
is called \emph{the operator system of the non-commuting $n$-cube.}

They showed that upto a $*$-isomorphism, $C^*_e (NC(n)) =
C^*(*_n\mathbb{Z}_2)$ (\cite[Corollary $5.6$]{disgrp}) and that
$NC(1)$ is (min,max)-nuclear (\cite[Proposition $6.1$]{disgrp}) and
$NC(2)$ is (min,c)-nuclear (\cite[ Theorem 6.3]{disgrp}). Further, it
follows from \cite[Theorem $6.13$]{disgrp} that $NC(n)$ is not (min,
max)-nuclear for all $n \geq 2$. \Cref{rigidity}$(iii)$ now provides
an alternate proof of this fact.

 Farenick et al. further introduced the {\em operator system of the
   commuting $n$-cube} as the operator subsystem $C(n) \subset
 C([-1, 1]^n)$ given by
$$
C(n) = \mathrm{span}\{ 1, x_1, x_2, \ldots, x_n \}
$$ where $x_i$ is the $i$-th coordinate function on $[-1,
  1]^n$. Clearly, the $C^*$-algebra generated by $C(n)$ in $ C([-1,
  1]^n)$ is commutative. As a concequence, the $C^*$-envelope of
$C(n)$ is also commutative and hence nuclear; in particular, $C(n)$ is
(min, ess)-nuclear.  

There was one more important example introduced in \cite{disgrp},
namely the operator system $\mcal{V} \subset \ell_4^\infty$ given by
\[\mcal{V} :=\{(a, b, c, d) : a +b = c+d\} \subset \ell^\infty_4.
\]
They proved in \cite[Theorem $6.11$]{disgrp} that $\mcal{V}$ is not
(min, max)-nuclear. It will be interesting to investigate the
following:

\vspace*{2mm}

\noindent {\bf Question:} Let $A$ be a unital commutative $C^*$-algebra
and $\mS$ be an operator subsytem of $A$. Is $\mS$ 
$C^*$-nuclear?

\subsection{ Operator system generated by a single operator}

Determining the $C^*$-envelope of a given operator system is in
general quite challenging. However, very recently, Argerami and
Farenick \cite{argerami1,argerami2} considered the operator systems
generated by single operators of certain classes of operators and
successfully calculated their $C^*$-envelopes. The calculation of
these $C^*$-envelopes together with \Cref{mainresult} lead to
interesting examples of finite dimensional (min, ess)-nuclear operator
systems. Further, using \Cref{rigidity}$(iii)$ and simple dimension
comparisons, we check whether these singly generated operator systems
are (min,max)-nuclear or not.
\vspace*{1mm}

The operator system generated by a bounded linear operator $T$ acting
on a complex Hilbert space $\mH$ is defined to be the unital
self-adjoint subspace $\mO\mS(T) = \text{span} \{ 1, T, T^*\} \subset
\B.$ Argerami and Farenick exploited the $\ast$-isomorphism between
$C^*_e(\mO\mS(T))$ and the quotient of $C^*(T)$ by the Silov boundary
ideal of $\mO\mS(T)$ given by Arveson \cite{arveson} in order to do
the explicit calculations of the $C^*$-envelopes. The following
follows from Remarks in \cite{argerami1}.

\begin{example}
\begin{enumerate}[(i)]
\item If $T$ is normal, then, $C^*_e(\mO\mS(T))$ is commutative and
  hence nuclear implying that $\mO\mS(T)$ is
  $\mathrm{(min, ess)}$-nuclear.
\item If $T $ is a contraction such that $\mathbb{T} \subset
  \sigma(T)$, then,
  $C^*_e(\mO\mS(T))=C(\mathbb{T})$ and hence
$\mO\mS(T)$ is a $\mathrm{(min, ess)}$-nuclear.
\item If $T$ is an isometry, then $\mO\mS(T)$ is $\mathrm{(min,
  ess)}$-nuclear. This is true because of (1) and the fact that
  $C^*_e(\mO\mS(T))=C(\mathbb{T})$ if
  $T$ is not unitary. 
\end{enumerate}
\noindent Since the dimensions of the operator systems $\mO\mS(T)$ in
$(1)$, $(2)$ and $(3)$, do not equal the dimensions of their
respective $C^*$-envelopes, $\mO \mS(T) \neq C^*_e(\mO\mS(T))$ in all
the three cases and, therefore, the above operator systems are are not
$(\min, \max)$-nuclear.
\end{example}

For the sake of convenience of the reader, we now recall the
definitions of the classes of operators whose operator systems were
considered in \cite{argerami1} and \cite{argerami2}.

\begin{example} \label{splcase} 
If $\C^\times :=\C \diagdown \{0\}$ and $\xi= (\xi_1, \xi_2, \ldots,
\xi_d) \in (\C^\times)^d$, then \emph{the irreducible weighted
  unilateral shift} with weights $\xi_1, \xi_2, \ldots, \xi_d $ is the
operator $W(\xi)$ on $\C^{d+1}$ given by the
matrix 
$$W(\xi)=\begin{bmatrix} 
0 & & & & 0\\ 
\xi_1 & 0 & & & \\ 
& \xi_2 & \ddots & & \\ 
& & \ddots & 0 & \\ 
& & & \xi_d & 0
\end{bmatrix}.$$

By \cite[Proposition 3.2]{argerami1}, $C^*_e(\mO\mS({W(\xi)}))=
M_{d+1}(\C)$, which is a nuclear $C^*$-algebra and hence
$\mO\mS(W(\xi))$ is $\mathrm{(min, ess)}$-nuclear.  Since
$M_{d+1}(\C)$ is atleast 4 dimensional, $\mO\mS(W(\xi)) \neq
C^*_e(\mO\mS({W(\xi)}))$, which further implies $\mO\mS(W(\xi))$ is
not $(\min, \max)$-nuclear.
\end{example}

\begin{example}
A {\em weighted unilateral shift operator} is an operator $W$ on
$l^2(\mathbb{N})$ whose action on the standard orthonormal basis
$\{e_n: n \in \mathbb{N} \}$ of $l^2(\mathbb{N})$ is given
by $$We_n=w_ne_{n+1}, \; n \in \mathbb{N},$$ where the weight sequence
$\{w_n\}_{n \in \N}$ for $W$ consists of non-negative real numbers
with $\sup_n w_n < \infty$. If there is a $p \in \N$ such that
$w_{n+p}=w_n$ for every $n \in \N$, then $W$ is called a periodic
unilateral weighted shift of period $p$. If at least one of $w_1,
\ldots , w_p$ is not repeated in the list, $W$ is said to be distinct.

By \cite[ Theorem 3.5]{argerami1}, $C^*_e( \mO\mS(W))=C(\mathbb{T}) \ot
M_p(\C),$ which is again a nuclear $C^*$-algebra and hence $\mO  \mS(W)$ is
$\mathrm{(min, ess)}$-nuclear and, clearly not $(\min, \max)$-nuclear.

Note that \Cref{splcase} is a special case of this.

\end{example}

Recall that an operator $J$ on an $n$-dimensional Hilbert space $\mH$
is a basic Jordan block if there is an orthonormal basis of $\mH$ for
which $J$ has a matrix representation of the form $$J_n(\lambda)
:= \begin{bmatrix} \lambda & 1 & 0 & \ldots & 0 \\ 0 & \lambda & 1 &
  \ddots & \vdots \\ \vdots & \ddots & \ddots & \ddots & 0\\ \vdots &
  & \ddots & \ddots & 1 \\ 0 & \ldots & \ldots & 0 & \lambda
\end{bmatrix}$$
for some $\lambda \in \C.$

\begin{example}
Let $J=\bigoplus_{k=1}^{\infty} J_{m_k}(\lambda) \in B(l^2(\N))$ and $m
:= sup\{m_k: k \in \N\}$. Then, 
by \cite[ Proposition 2.2]{argerami2}, we have \[C^*_e(\mO\mS(J))=\begin{cases}
      C(\mathbb{T}), & \text{if}\;m=\infty; \ \text{and} \\
      M_m(\C), & \text{if}\;m<\infty,
          \end{cases}
\] and hence
$\mO\mS(J)$ is $\mathrm{(min, ess)}$-nuclear. As in earlier examples,
one can see that for $m > 1$, $\mO\mS(J)$ is not $(\min, \max)$-nuclear.
\end{example}

An operator $J \in B(l^2(\N))$ is said to be a {\em Jordan operator}
if $J= \bigoplus_j J_{n_j}({\lambda_{j}})$ for some finite or infinite
sequence of basic Jordan operators $J_{n_j}(\lambda_j).$ In this
definition, the $n_j$'s or the $\lambda_j$'s are not required to be
distinct. But repetitions of the same pair $n_j , \lambda_j$ are not
allowed: if a direct sum of $d$ copies of a basic Jordan block $J_n
(\lambda)$ is considered, then it is denoted by $J_n (\lambda) \otimes
1_d$.
 
\begin{example} Consider the Jordan operator $J= \begin{bmatrix}
1 & & & & \\
 & \omega & & & \\
 & & \omega^2 & & \\
 & & & 0 & 1\\
 & & & 0 & 0
\end{bmatrix},$
where $\omega= (-1 -i\sqrt{3})/2$. Then, by \cite[Remark
  2.7]{argerami2}, $C^*_e(\mO\mS(J))= \C^3$, which shows that
$\mO\mS(J)$ coincides with its $C^*$-envelope and hence, is
$\mathrm{(min, max)}$-nuclear.
\end{example}

\begin{example}
If $J=\bigoplus_{k=1}^{n}( J_{m_k}({\lambda_{k}}) \ot I_{d_k})$, with
$\lambda_1 > \lambda_2 > \cdots > \lambda_n$ all real and $\max\{m_2,
\ldots, m_{n-1}\} \leq \min\{m_1, m_n\}$, then, by \cite[Corollary
  2.12]{argerami2},  $C^*_e(\mO\mS(J))$ equals
\[ 
\begin{cases}
M_{m_1}(\C) \oplus M_{m_n}(\C), & \text{if} \; \min\{m_1,m_2\} \geq 2;\\
\C \oplus \C, & \text{if} \ m_1 = m_n=1; \\
M_{m_n}(\C), & \text{if} \; m_1 = 1, m_n \geq 2, |\lambda_1 -\lambda_n| \leq \cos \frac{\pi}{(m_n+1)};\\
\C \oplus \M_{m_n}(\C), & \text{if} \ m_1 = 1, m_n \geq 2, |\lambda_1 -\lambda_n| > \cos \frac{\pi}{(m_n+1)};\\
M_{m_1}(\C), & \text{if} \ m_1 \geq 2, m_n = 1 , |\lambda_1 -\lambda_n| \leq \cos \frac{\pi}{(m_1+1)};\;  \\
M_{m_1}(\C) \oplus \C, & \text{if} \ m_1 \geq 2, m_n =1 , |\lambda_1 -\lambda_n| > \cos \frac{\pi}{(m_n+1)},
\end{cases} 
\]
implying that $\mO\mS(J)$ is $\mathrm{(min, ess)}$-nuclear but not $(\min, \max)$-nuclear for all of the
above cases except for the case when  $m_1 = m_n=1$. If $m_1 = m_n=1$, $\mO\mS(J)$ turns out to be $(\min, \max)$-nuclear.

\end{example}


\subsection*{Acknowledgements}
{\small The first named author would like to thank Professors Gilles
  Pisier and Vern I.~Paulsen for introducing him to the worlds of
  Operator Spaces and Operator Systems at a workshop held at IMSc,
  Chennai, and Professor V.~S.~Sunder for organizing that workshop.
  The second named author would like to express her sincere gratitude
  to Professor Ajay Kumar for his kind efforts in acquainting her with
  all the tools of Operator Algebras and literature available with
  him; and would also like to thank Professor ~K.~H.~Han for several
  fruitful discussions over email. The authors would also like to
  thank the referee for the constructive comments that helped to improve the manuscript.}

\bibliographystyle{plain}
\bibliography{paperref}

\end{document}